\documentclass[draft, reqno]{amsart}
\usepackage[all]{xy}                        %

\CompileMatrices                            

\UseTips                                    

\input xypic
\usepackage[bookmarks=true]{hyperref}       

\usepackage{amssymb,latexsym,amsmath,amscd}
\usepackage{xspace}
\usepackage{enumerate}
\usepackage{graphicx}
\usepackage{vmargin}
\usepackage{todonotes}

\DeclareMathOperator{\red}{red}

\reversemarginpar

\theoremstyle{plain}
\newtheorem{theorem}{Theorem}[section]
\newtheorem*{theorem*}{Theorem}
\newtheorem{proposition}[theorem]{Proposition}

\newtheorem{lemma}[theorem]{Lemma}

\newtheorem{question}[theorem]{Question}

\theoremstyle{definition}
\newtheorem{definition}[theorem]{Definition}

\newtheorem{remark}[theorem]{Remark}
\newtheorem{example}[theorem]{Example}

\newcommand{\enm}[1]{\ensuremath{#1}}          %

\newcommand{\cal}[1]{\mathcal{#1}}

\newcommand{\CC}{\enm{\mathbb{C}}}

\newcommand{\FF}{\enm{\mathbb{F}}}

\newcommand{\PP}{\enm{\mathbb{P}}}

\newcommand{\Aa}{\enm{\cal{A}}}
\newcommand{\Bb}{\enm{\cal{B}}}

\newcommand{\Ii}{\enm{\cal{I}}}

\newcommand{\Oo}{\enm{\cal{O}}}

\newcommand{\Ss}{\enm{\cal{S}}}

\renewcommand{\phi}{\varphi}
\renewcommand{\theta}{\vartheta}
\renewcommand{\epsilon}{\varepsilon}

\begin{document}

\title[Segre varieties]
{Linear dependent subsets of Segre varieties}
\author{E. Ballico}
\address{Dept. of Mathematics\\
 University of Trento\\
38123 Povo (TN), Italy}
\email{ballico@science.unitn.it}
\thanks{The author was partially supported by MIUR and GNSAGA of INdAM (Italy).}
\subjclass[2010]{14N05; 12E99; 12F99}
\keywords{Segre varieties}

\begin{abstract}
We study the linear algebra of finite subsets $S$ of a Segre variety $X$. In particular we classify the pairs $(S,X)$
with $S$ linear dependent and $\#(S)\le 5$. We consider an additional condition for linear dependent sets (no two of their
points are contained in a line of
$X$) and get far better lower bounds for $\#(S)$ in term of the dimension and number of the factors of $X$. In this discussion
and in the classification of the case $\#(S)= 5$, $X\cong \PP^1\times \PP^1\times \PP^1$ we use the rational normal curves
contained in $X$.
\end{abstract}

\maketitle

\section{Introduction}

Let $K$ be a field. Fix positive integers $k$ and $n_i$, $1\le i\le k$. Set $Y:= \prod _{i=1}^{k} \PP^{n_i}$ (the
multiprojective space with $k$ non-trivial factors of dimension $n_1,\dots ,n_k$). Set $r:= -1+ \prod _{i=1}^{k} (n_i+1)$. 
Let $\nu : Y\to \PP^r$ denote the Segre embedding of the multiprojective space $Y$. Thus $X:= \nu (Y)$ is a Segre variety of
dimension $n_1+\cdots +n_k$. It was introduced by Corrado Segre in 1891 (\cite{seg}). See \cite[Ch. 25]{ht} for its geometry over a
finite field. In the last 30 years this variety had a prominent role in the applied sciences, because it is strongly related
to tensors and it was realized that tensors may be used in Engineering and other sciences (\cite{l}).

Let $S\subset Y$ be a finite subset. Set
$e(S):= h^1(\Ii _S(1,\dots ,1)) = \# (A) -1 -\dim \langle \nu (A)\rangle$, where $\langle \ \ \rangle$ denote the linear
span. The minimal multiprojective subspace
$Y'$ of
$Y$ containing
$S$ is the multiprojective space $\prod _{i=1}^{k} \langle \pi _i(S)\rangle \subseteq Y$, where  $\langle \pi _i(S)\rangle$
denote the linear span of the finite set $\pi _i(S)$ in the projective space $\PP^{n_i}$. We say that $Y'$ is the multiprojective subspace generated by $S$ and that $S$ is \emph{nondegenerate} if $Y'=Y$. We say that
$S$ is
\emph{linearly independent} if
$\nu (S)\subset \PP^r$ is linearly independent. By the definitions of Segre embedding and of the integer $e(S)$ we have $\dim
\langle
\nu (S)\rangle =\# (S)-1-e(S)$. In particular $S$ is linearly dependent if and only if $e(S)>0$. We say that
$S$ is a
\emph{circuit} if
$S$ is linearly dependent, but every proper subset of $S$ is linearly independent.

 Everything said up to now use only the
linear structure of the ambient $\PP^r$. Now we describe the new feature coming from the structure  of $Y$ as a
multiprojective space, in particular the structure of linear subspaces contained in the Segre variety $\nu (Y)$. We say that a
finite set $S\subset Y$ is \emph{minimal} if there is no line $L\subset \nu (Y)$ such that $\# (\nu (S)\cap L) \ge 2$. Of
course, if $\# (\nu (S)\cap L) \ge 3$, then $S$ is not linearly independent. However, a non-minimal finite set $S$ may be
linearly independent (take as $S$ two points such that the line $\langle \nu (S)\rangle$ is contained in $\nu (Y)$). When $S$
is linearly independent there is no $A\subset Y$ such that $\# (A)<\# (S)$ and $\langle \nu (A)\rangle \supseteq
\langle
\nu (S)\rangle$, but if $S$ is not minimal, say there is $L\subset Y$ such that $\nu (L)$ is a line and $L\cap S\supseteq
\{a,b\}$ with $a\ne  b$,  there is $o\in L$ such that $q\in \langle \nu (o) \cup (S\setminus \{a,b\})$. Note that $\# (\{o\}\cup (S\setminus \{a,b\}) <\# (S)$. We say that $S$ is \emph{$i$-minimal} if there is no curve $J\subset Y$ such that $\nu (J)$ is a line, $\# (J\cap S)\ge 2$ and $J$ is mapped isomorphically into the $i$-th factors of $Y$, while
it is contracted to a point by the projections onto the other factors of $Y$. The finite set $S$ is minimal if and only if it
is $i$-minimal for all $i$. The minimality condition is in general quite weaker/different from the assumptions needed to apply
the famous Kruskal's criterion to two subsets $A, B\subset S$ with $A\cup B=S$ and $\# (A) = \# (B)$
(\cite{co,cov,cov2,ddl1,ddl2,ddl3,kr}).

We classify  circuits with cardinality $4$ (Proposition \ref{e2}) and give the following classification of circuits formed by $5$ points.
\begin{theorem}\label{e3}
Let $\Sigma$ denote the set of all nondegenerate circuits $S\subset Y$ such that $\# (S) =5$. Then one of
the following cases occurs:
\begin{enumerate}
\item $k=1$, $n_1=3$;
\item $k=2$, $n_1=n_2=1$;
\item $k=2$ and $n_1+n_2 =3$; all $S\in \Sigma$ are described in Example \ref{p2p1};
\item $k=3$, $n_1=n_2=n_3=1$; all $S\in \Sigma$ are described in Lemma \ref{p1p1p1}; in this case $\Sigma$ is an irreducible
variety of dimension $11$.
\end{enumerate}
\end{theorem}

All $S\in \Sigma$ in the first two  cases listed in Theorem \ref{e3} are obvious and we describe them in Remark
\ref{e3.0}. In the classification of case $k=3$ we use the rational normal curves contained in a Segre variety. 

We also
classify the nondegenerate sets
$S$ with
$\# (S)=5$ and
$e(S)\ge 2$ (Proposition
\ref{e3.01}).

The study of
linearly dependent subsets of Segre varieties with low cardinality was started in
\cite{sac}.

The author has no conflict of interest.

\section{Preliminaries}
Take $Y =\PP^{n_1}\times \cdots \times \PP^{n_k}$, $k\ge 1$, $n_i>0$, $1\le i\le k$. Set $r:= -1 + \prod _{i=1}^{k} (n_i+1)$. Let $\nu: Y \to \PP^r$
denote the Segre embedding of $Y$. We will use the same name, $\nu$, for the Segre embedding of any
multiprojective subspace $Y'\subseteq Y$. Set
$X:=
\nu (Y)\subset \PP^r$. For any $q\in  \PP^r$ the $X$-rank of $q$ is the minimal cardinality of a finite subset $S\subset X$
such that $q\in \langle S\rangle$, where $\langle \ \ \rangle$ denote the linear span. For any $q\in \PP^r$ let $\Ss (Y,q)$
denote the set
of all $A\subset Y$ such that $\# (A) =r_X(q)$ and $q\in \langle \nu (A)\rangle$. In the introduction we observed that $S\notin \Ss (Y,q)$ for any $q\in \PP^r$ if $S$ is not minimal.

For any $i\in \{1,\dots ,k\}$ set
$Y_i:= \prod_{h\ne i} \PP^{n_h}$ with the convention that $Y_1$ is a single point if $k=1$. 

Let $\eta _i: Y\to Y_i$ denote
the projection (it is the map forgetting the $i$-th coordinate of the points of $Y$). We have $h^0(\Oo _Y(1,\dots ,1)) =r+1$. For any $i\in \{1,\dots ,k\}$ let $\Oo _Y(\epsilon _i)$ (resp. $\Oo
_Y(\hat{\epsilon}_i)$) be the line bundle $\Oo _Y(a_1,\dots ,a_k)$ on $Y$ with multidegree $(a_1,\dots ,a_k)$ with $a_i=1$ and
$a_j =0$ for all $j\ne i$ (resp. $a_i=0$ and $a_j=1$ for all $j\ne i$. We have $h^0(\Oo _Y(\epsilon _i)) =n_i+1$ and $h^0(\Oo
_Y(\hat{\epsilon}_i)) =(r+1)/(n_i+1)$.

\begin{definition}
Take $S\subset Y$ such that $e(S)>0$. We say that $S$ is \emph{strongly essential} if $e(S')=0$ for all $S'\subset S$ such that
$\# (S') =\# (S)-e(S)$.
\end{definition}

Take an essential set $S\subset Y$ and any $S'\subset S$. We have $e(S') =\max \{0,e(S)-\# (S)+\# (S')\}$.

A set $S\subset Y$ with $e(S)=1$ is strongly essential if and only if it is a circuit.

We recall the following lemma (\cite[Lemma 2.4]{bbcg1}), whose proof works over any algebraically closed field, although it was only claimed over $\CC$, or at least over an algebraically closed base field
with characteristic $0$. 

\begin{lemma}\label{ee0}
Take $q\in \PP^r$ and and finite sets $A, B\subset Y$ irredundantly spanning $q$. Fix an effective divisor $D\subset Y$.
Assume $A\ne B$ and $h^1(\Ii _{A\cup B\setminus D\cap (A\cup B)}(1,\dots ,1)(-D))=0$. Then $A\setminus A\cap D =B\setminus
B\cap D$.
\end{lemma}
\begin{proof}
In \cite{bbcg1} there is the default assumption that the base field is $\CC$ (or al least an algebraically closed field with
characteristic $0$). The proof of \cite[Lemma 2.5]{bbcg1} (whose statement imply \cite[Lemma 2.4]{bbcg1}) never uses any
assumption on the characteristic of the base field. Now we explain why the statement of Lemma \ref{ee0} over an algebraic
closure $\overline{K}$ of $K$ implies the statement over $K$. By assumption all points of $A$ and $B$ are defined over $K$.
The dimension of a linear span of a subset of $\nu(A\cup B)$ is the same over $K$ or over $\overline{K}$. Since the
statement of the lemma also uses cohomology groups of coherent sheaves we also need to use that the dimension of the cohomology
groups of coherent sheaves on projective varieties defined over $K$ is preserved when we extend the base field $K\subseteq
\overline{K}$, because $\overline{K}$ is flat over $K$ (\cite[Proposition III.9.3]{h}).
\end{proof}

\begin{remark}Fix an integer $e>0$ and an integral and non-degenerate variety $W\subset \PP^r$ defined over $\overline{K}$. Since
$W(\overline{K})$ is Zariski dense, $r+1+e$ is the maximal cardinality over a finite set $S\subset W(\overline{K})$ such that
$\dim \langle S\rangle =\# (S)-1-e$. The same is true if $W$ is defined over $K$ and we require that $S\subset W(K)$ and
that $W(K)$ is Zariski dense in $W(\overline{K})$. The minimal such cardinality of any such set $S\subset W(\overline{K})$
is
$e+2$ if and only if $W$ contains a line; otherwise it is larger. If $K$ is finite to get the same we need $\# (K) \ge e+1$
and that the line $L\subset W$ is defined over $K$. The Segre variety has plenty of lines defined over the base field $K$.
\end{remark}

\section{Linear algebra inside the Segre varieties}\label{Sl}

Take $Y:= \PP^{n_1}\times \cdots \times \PP^{n_k}$, $n_i>0$ for all $i$. 

\begin{remark}Fix a finite nondegenerate set $S\subset Y$.
We assume $h^1(\Ii _S(1,\dots ,1)) >0$ (i.e., that $\nu (S)$ is linearly independent, i.e., (with our terminology) that $S$ is linearly dependent) and $h^1(\Ii _{S'}(1,\dots ,1)) =0$ for all $S'\subsetneq S$ (i.e., that each proper subset of $S'$ is linearly independent). Equivalently, let $S\subset Y$ be a nondegenerate circuit. In particular we have $h^1(\Ii _S(1,\dots ,1)) =1$, i.e., $\dim \langle \nu (S)\rangle = \# (S)-2$. Since $\Oo _Y(1,\dots ,1)$ is very ample, we have $\# (S)\ge 3$.
\end{remark}
\begin{remark}\label{eo0}
Take $Y = \PP^{n_1}\times \times \cdots \times \PP^{n_k}$ and set $m:= \max \{n_1,\dots ,n_k\}$. It easy to check that $m+1$
is the minimal cardinality of a subset of $Y$ generating $Y$, i.e., not contained in a proper multiprojective subspace of $Y$ (just take
$S$ such
that $\dim \langle \pi _i(S)\rangle =\min \{n_i,\# (S)-1\}$ for all $i$).
\end{remark}

\begin{example}\label{ee1}
Let $S\subset Y$ be a finite linearly independent subset, $S\ne \emptyset$, and set $s:= \# (S)$. Fix
$i\in
\{1,\dots ,k\}$. We construct another subset $S_i\subset Y$ such that $\# (S_i) =s+1$ and $\#
(S_i\cap S) =s-1$ in the following way and discuss when (assuming $s\le r$) $S_i$ is linearly independent. Fix $o\in S$ and $i\in \{1,\dots ,k\}$. Take a line $L\subseteq \PP^{n_i}$ containing
$o_i$. Fix two points $o'_i,o''_i\in L\setminus L\cap \pi _i(S)$. Let $o'$ (resp. $o''$) be the only point of $Y$ with $\pi
_i(o') = o'_i$ (resp. $\pi _i(o'') = o''_i$) and $\pi _j(o')=\pi _j(o'') = \pi _j(o)$ for all $j\ne i$. Set $S_i:= (S\setminus
\{o\})\cup \{o',o''\}$. We have $\# (S_i) =\# (S)+1$. By assumption we have $\dim \langle \nu (S)\rangle =s-1$. Let
$D\subseteq D'\subseteq Y$  be the multiprojective subspace of $Y$ with $\pi _j(o)$ as their projection for all $j\ne i$, $\pi
_i(D)=L$ and $\pi _i(D') =\PP^{n_i}$. The set $S_i$ is linearly independent for general $o'_i,o''_i\in L$ (resp. for general
$o'_i,o''_i\in L$ and for a general line $L\subseteq \PP^{n_i}$) if and only if $\dim \langle \nu (S\cup D)\rangle \ge s$
(resp.  $\dim \langle \nu (S\cup D)\rangle \ge s$). Now assume $s\ne r+1$ and set $E:= S\setminus \{o\}$. Take a general $(a_1,\dots ,a_k)\in Y$, and a general
line $J\subset \PP^{n_i}$ such that $a_i\in J$. Let $T\subset Y$ be the irreducible curve with $\pi _h(T)=a_h$ for all $h\ne i$, $\pi _i(T)=J$ and $\pi _{i|T}: T\to J$ an isomorphism.
Then $L:= \nu (T)$ is a line and $\dim \langle L\cup \nu (E)\rangle =s$. We describe in Lemma \ref{ee1.1} and Remark
\ref{ee1.12} some cases in which we may take $\nu (o)\in L$. Of course, for any linearly independent set $A\subset Y$ with
$\# (A) \le r$, there is $o\in Y\setminus A$ such that $A\cup \{o\}$ is linearly independent, because $\nu (Y\setminus A)$
spans
$\PP^r$.
\end{example}

The set $S_i$ constructed in Example \ref{ee1} will be called an \emph{elementary increasing} of $S$ in the $i$-th factor or an \emph{$i$-elementary increasing} of $S$. We say that the set $S$ is obtained from $S_i$ by an \emph{elementary decreasing} of the non-minimal set $S_i$ in the $i$-th direction or by an \emph{$i$-elementary decreasing}.

Note that $S_i\subset S$ is obtained from some $i$ making an elementary increasing along the $i$-th component if and only if
$\eta _{i|S_i}$ is not injective.

\begin{remark}\label{ee1.0}
Take  $S, S_i, o, o', o'' $ as in Example \ref{ee1}. Since $\nu (o)\in \langle \nu (\{o',o''\})\rangle$, we have $\langle \nu (S)\rangle\subseteq \langle \nu (S_i)\rangle$, with strict inclusion if and only if $S_i$ is linearly independent.
\end{remark}

\begin{lemma}\label{ee1.1}
Let $S\subset Y$ be a linearly independent finite subset, $S\ne \emptyset$, such that there is $i\in \{1,\dots ,k\}$ such that $\langle \pi _i(S)\rangle \subsetneq \PP^{n_i}$. Then $S$ has a linearly independent elementary increasing in the $i$-th direction.
\end{lemma}

\begin{proof}
Set $s:= \# (S)$. Let $Y'\subset Y$ be the multiprojective space with $\PP^{n_j}$ as its factor for all $j\ne i$ and $\langle \pi _i(S)\rangle$ as its $i$-th factor. By assumption we have $Y'\subsetneq Y$. Fix $o\in S$. A general line $L\subset \PP^{n_i}$ containing $\pi _i(o)$ is not contained in $\langle \pi _i(S)\rangle$. Call $Y''\subseteq Y$
the multiprojective space with $\PP^{n_j}$ as its factor for all $j\ne i$ and $\langle \pi _i(S)\cup L\rangle$ as its $i$-th factor. We have $S\subset Y'\subsetneq
Y''\subseteq Y$. Hence
$\nu (L)\nsubseteq \langle \nu (Y')\rangle$. Fix any $o'_i,o''_i\in \pi _i(L)\setminus \{\pi _i(o)\}$ such that $o'_i\ne
o''_i$. Take $o', o''\in Y$ with $\pi _i(o')=o'_i$, $\pi _i(o'')=o''_i$ and $\pi _j(o')=\pi _j(o'')=o_j$ for all $j\ne i$.
\end{proof}

\begin{remark}\label{ee1.12}
For each $a
=(a_1,\dots ,k\}$ set $a[i]:= \eta _i^{-1}(\eta _i(a))$. Note that $a[i] \cong \PP^{n_i}$ and that $\nu (a[i])$ is an
$n_i$-dimensional linear space containing $\nu (a)$. Set $\{\{a\}\}:= \cup _{i=1}^{k} a[i]$. Note that $\dim
\langle\nu(\{\{a\}\})\rangle = n_1+\cdots +n_k$. For any finite set $A\subset Y$, $A\ne\emptyset$, set $A[i]:= \cup _{a\in A}
a[i]$ and $\{\{A\}\}:= \cup _{a\in A} \{\{a\}\}$. Now assume that $A$ is linearly independent and $\# (A)\le r$. Since any
two points of a projective space are contained in a line, $A$ has a linearly independent $i$-increasing (resp a linearly
independent increasing) if and only if $\langle \nu (A)\rangle \subsetneq \langle \nu (A[i])\rangle$ (resp.  $\langle \nu
(A)\rangle
\subsetneq \langle \nu (\{\{A[i])\}\}\rangle$). By Lemma \ref{ee1.1} to prove that $A$ has a linearly independent elementary
increasing we may assume that $\langle \pi _i(A)\rangle =\PP^{n_i}$ for all $i$. If $k=2$ and $\langle \pi _i(A)\rangle
=\PP^{n_i}$ for at least one $i$, then it is easy to see that $\langle \nu (A)\rangle =\PP^r$.
\end{remark}

\begin{remark}\label{p2}
Fix a linearly independent $S\subset Y$ and take  $i\in \{1,\dots ,k\}$ and $a\in Y_i$. Since $3$ collinear points
are not
linearly independent, we have $\# (S\cap \eta _i^{-1}(a)) \le 2$. Take a circuit $A\subset Y$. Either $\# (A\cap \eta
_i^{-1}(a)) \le 2$ or $A$ is formed by $3$ collinear points (and so $\PP^1$ is the minimal multiprojective space containing
$A$).
\end{remark}

Let $S\subset Y$ be a finite set such that $e(S)>0$. Note that $\# (S)\ge e(S)+2$. A point $o\in S$ is said to be
\emph{essential for $S$} if $e(S\setminus \{o\}) = e(S)-1$. If $o$ is not essential for $S$ we will often say that
$o$ is \emph{inessential for $S$}. Since $e(S)-e(S') \le \# (S)-\# (S')$ for all $S'\subset S$, $o$ is inessential for
$S$ if and only if $e(S\setminus \{o\}) = e(S)$.

Let $S\subset Y$ be a finite set such that $e(S)>0$. A \emph{kernel} of $S$ is a minimal subset $S'\subseteq S$ such that
$e(S')=e(S)$. 

\begin{lemma}\label{p3}
Any finite linearly dependent subset of $Y$ has a unique kernel.
\end{lemma}

\begin{proof}
Take a finite set $S\subset Y$ such that $e(S)>0$. Let $S'$ and $S''$ be kernels of $S$. Assume $S'\ne S''$. By the definition of
kernels we have $S''\nsubseteq S'$. We order the points of $S''$, say $ S'' =\{o_1,\dots ,o_b\}$ so that $o_b\notin S'$. By the
definition of kernel we have
\begin{equation}\label{eqp1}
\dim
\langle
\nu (S)\rangle =\dim \langle \nu (S')\rangle+\# (S\setminus S'). 
\end{equation}
We first add $\{o_1,\dots ,o_{b-1}\}$ to $S'$ and get a set
$S_1\supseteq S'$. By (\ref{eqp1}) we get $\dim \langle \nu (S_1)\rangle =\dim \langle \nu (S')\rangle +b-1 -\# (S'\cap
S'')$. Then we add $o_b$. Since $e(S'')>0$, we have $\nu (o_b)\in \langle \nu (\{o_1,\dots ,o_{b-1}\}\rangle$ and hence
$\nu (o_b)\in \langle \nu (S_1)\rangle $, contradicting (\ref{eqp1}).
\end{proof}

\begin{lemma}\label{p4}
Let $S\subset Y$ be a finite subset such that $e(S)>0$. The kernel of $S$ is the set of all its essential points, i.e., the
tail of $S$ is the set of all its inessential points.
\end{lemma}

\begin{proof}
Fix an inessential point $o\in S$ (if any). Let $S'\subset S\setminus \{o\}$ be a minimal subset of $S\setminus \{o\}$ with
$e(S')=e(S)$. By Lemma \ref{p4} $S'$ is the unique kernel of $S$. Hence the tail of $S$ contains $o$. Thus the tails of $S$
contains all inessential points of $S$. No essential point of $S$ may belong to the tail.
\end{proof}

Let $S\subset Y$ be a finite subsets with $e(S)>0$. The \emph{tail} of $S$ is $S\setminus S'$, where $S'$ is the kernel of
$S$. The tail of $S$ is the set of all inessential points of
$S$, while the \emph{kernel} of $S$ is the set of all its essential points.

\begin{remark}\label{p3.00}
Let $A\subset \PP^m$ be a linearly dependent finite subset, i.e., a finite subset such that $e(A):= \# (A)-1-\dim \langle A\rangle >0$. We first show to different methods
to get a subset $B\subseteq A$ with $e(B) =e(A)$, $e(B')$. Easy examples (even when $m=1$) shows that each of these methods
does not give a unique $B$. The first method is increasing the number of points in a subset of $A$. We start with $o, o'\in A$
such that $o\ne o'$. Obviously $e(\{o,o'\})=0$
\end{remark}
 
Consider the following construction: linear projections of a multiprojective space from proper linear subspaces of
one of its factors. Fix integers $0\le v < n$ and a $v$-dimensional linear subspace $V\subset \PP^n$. Let $\ell _V:
\PP^n\setminus V\to \PP^{n-v-1}$ denote the linear projection from $V$. Take $Y = \PP^{n_1}\times \cdots \times \PP^{n_k}$.
Fix an integer
$i\in \{1,\dots ,k\}$, an integer $v$ such that $0\le v <n_i$ and a $v$-dimensional linear subspace $V\subset \PP^{n_i}$. Let
$Y'$ (resp. $Y''$) be the multiprojective space with $k$ factors, with $\PP^{n_h}$ as its $h$-th factor for all $h\ne i$ and
with $V$ (resp. $\PP^{n_i-v-1}$) as its $i$-th factor. The multiprojective space $Y'$ (resp. $Y''$) has $k$ non-trivial
factors if and only if $v>0$ (resp. $v\le n_i-2$). Let $\ell _{V,i}: Y\setminus Y' \to Y''$ denote the morphism defined by the
formula $\ell _{V,i}(a1,\dots ,a_k) = (b_1,\dots ,b_k)$ with $b_h=a_h$ for all $h\ne i$ and $b_i =\ell _V(a_i)$. When $K$ is
infinite we use the Zariski topology on $\PP^n(K)$ and the $K$-points of the Grassmannians.

\begin{remark}\label{f3}
Fix a finite set $S\subset Y$, an integer $i\in \{1,\dots ,s\}$ and an integer $v$ such that $0\le v\le n_i-2$. Assume for the
moment that
$K$ is infinite. Let $V$ be a general (for the Zariski topology) $v$-dimensional linear subspace. Then $\ell _{V,i|S}$ is
injective. For fixed
$k, n_1,\dots ,n_k$ there is an integer $q_0$ such that for all $q\ge 2$ there is a $v$-dimensional linear subspace $V\subset
\PP^{n_i}(\FF _q)$ such that $\ell _{V,i|S}$ is injective. When $\ell _{V,i|S}$ is injective, we obviously have $h^1(Y,\Ii
_S(1,\dots ,1)) \le h^1(Y'',\Ii _{\ell _{V,i}(S)}(1,\dots ,1))$. If $Y$ is the minimal multiprojective subspace containing
$S$, then $Y''$ is the minimal multiprojective space containing $\ell _{V,i}(S)$.
\end{remark}

\section{Rational normal curves inside a Segre variety}
Fix positive integers $k$ and $n_i$, $1\le i\le k$. Set $Y:= \PP^{n_1}\times \cdots \times \PP^{n_k}$. Let $\Bb (n_1,\dots
,n_k)$
denote the set of all integral curves $D\subset \PP^1$ such that $D = h(\PP^1)$ with $h =(h_1,\dots ,h_k): \PP^1\to Y$ with
$h_i: \PP^1\to \PP^{n_i}$ an embedding with $h_i(\PP^1)$ a rational normal curve of $\PP^{n_i}$. The set $\Bb (n)$ of all
rational normal curves of $\PP^n$ is a rational variety of dimension $n(n+3)$. Thus $\Bb (n_1,\dots ,n_k)$ is parametrized
by an irreducible variety. For any $D\in \Bb (n_1,\dots ,n_k)$ we have $\dim \langle \nu (D)\rangle = n_1+\cdots +n_k$
and $\nu (D)$ is a degree $n_1+\cdots +n_k$ rational normal curve of $\langle \nu (D)\rangle$. Let
$D\subset Y$ be a curve. Obviously
$D\in
\Bb (n_1,\dots ,n_k)$ if and only if the following conditions are satisfied:
\begin{itemize}
\item[(a)] $D$ is an integral curve;
\item[(b)] $\pi _{i|D}$ is birational onto its image for all $i=1,\dots ,k$;
\item[({c})] $p_i(D)$ is a rational normal curve of $\PP^{n_i}$ for all $i=1,\dots ,k$.
\end{itemize}
 Obviously $D\in \Bb (n_1,\dots ,n_k)$ if and only if the following
conditions are satisfied:
\begin{itemize}
\item[($a_1$)] $D$ is an integral curve;
\item[($b_1$)] $\deg (\nu (D)) =n_1+\cdots +n_k$;
\item[($c_1$)] $Y$ is the minimal multiprojective subspace of $Y$ containing $D$.
\end{itemize}

\begin{remark}
The integer $n_1+\cdots +n_k$ is the minimal degree of a connected and reduce curve $\nu (D)$, $D\subset Y$.
\end{remark}

\begin{remark}\label{n1}
Fix $D\in \Bb (n_1,\dots ,n_k)$ and any finite subset $S\subset D$. Since $\dim \langle \nu (D)\rangle = n_1+\cdots +n_k$
and $\nu (D)$ is a rational normal curve of $\langle \nu (D)\rangle$, we have
$$e(S) = \max \{0,\# (S)-n_1-\cdots -n_k-1\}.$$
Since $\# (L\cap \nu (D))\le 1$ for each $L\subset \nu (Y)$, $S$ is minimal.
\end{remark}

\begin{remark}\label{up1}
Take $Y:= (\PP^1)^k$. Let $T\subset Y$ be an integral curve. The \emph{multidegree} $(a_1,\dots ,a_k)$ of $T$ is defined in the
following way. If $\pi _i(T)$ is a point, then set $a_i:= 0$. If $\pi _i(T)=\PP^1$ let $a_i$ be the degree of the morphism
$\pi _{i|T}: T\to \PP^1$. If $k=3$, we say \emph{tridegree} instead of multidegree. Note that if $a_i>0$ for all $i$, then $T$
is not contained in a proper multiprojective subspace of $Y$. If $a_i=1$ for some $i$, then $\pi _{i|T}: T \to \PP^1$ is a
degree $1$ morphism between integral curves with the target smooth. By Zariski's Main Theorem (\cite{h}) $\pi _{i|T}$ is an
isomorphism and in particular $T\cong \PP^1$. Let $\Bb _k$ denote the set of all $T\subset Y$ with multidegree $(1,\dots ,1)$.
We just say that for any $T\in \Bb _k$, $T\cong \PP^1$ and each $\pi _{i|T}: T\to \PP^1$ is an isomorphism. Thus $\Bb _k$ (as
algebraic set) is isomorphic to $\mathrm{Aut}(\PP^1)^k$. We have $\Bb _k =\Bb (1,\dots ,1)$.
\end{remark}

 Obviously $D\in \Bb _k$ if and only if the following
conditions are satisfied:
\begin{itemize}
\item[($a_2$)] $D$ is an integral curve;
\item[($b_2$)] $D$ has multidegree $(1,\dots ,1)$.
\end{itemize}

\begin{example}\label{n2}
Fix positive integers $e$, $k$ and $n_i$, $1\le i\le k$, such that $k\ge 2$. Set $Y = \PP^{n_1}\times \cdots \PP^{n_k}$
and $m:= \max \{n_1-1,n_2,\dots ,n_k\}$. Fix
a line $L\subseteq \PP^{n_1}$ and $e+2$ points $a_1,\dots ,a_{e+2}\in L$. Fix $o_i\in \PP^{n_i}$, $2\le i\le k$.
Let $b_i\in Y$, $1\le i\le e+2$, be the point of $Y$ with $\pi _1(b_i)=a_i$ and $\pi _h(b_i) =o_h$ for all $h\in \{2,\dots
,k\}$.
Fix any $A\subset Y$ such that $\# (A)=m$, $\pi _1(A)\cup L$ spans $\PP^{n_1}$ and $\pi _h(A)$ spans $\PP^{n_h}$ for all $h\in \{2,\dots
,k\}$. Set $S:= A\cup \{b_1,\dots ,b_{e+2}\}$. We have $\# (S) =e+2+m$, $S$ is not contained in a proper multiprojective
subspace of $Y$ and $e(S)=e$. 
\end{example}

\begin{proposition}\label{n3}
Fix positive integers $k$ and $n_i$, $1\le i\le k$, such that $k\ge 2$ and $n_i\le n_1$ for all $i$.
Let $S\subset Y$ be a finite subset of $Y$ such that $e(S) >0$ and $S$ is not contained in a proper multiprojective subspace of
$Y$. Then

$$\# (S) \ge e(S)+n_1+1.$$ 
\end{proposition}

\begin{proof}
Since the proposition is trivial if $k=1$ we may assume $k\ge 2$ and that the proposition is true for multiprojective spaces with smaller dimension. Up to a permutation of the factors of $Y$ we may assume $\Aa =\{1,\dots ,c\}$. We first consider sets $S$ with $e(S)=1$. Set $s:= \# (S)$. Fix $E\subseteq S$ such that $\# (E)=\min \{n_1,s\}$. Since $h^0(\Oo _Y(\epsilon _1))=1$,

Take $H\in |\Oo _{Y'}(\epsilon _1)|$ containing at least $m_1$ points of $S$. Since $S\nsubseteq H$, Lemma \ref{n3} gives
$h^1(\Ii _{S\setminus S\cap H}(\hat{\epsilon}_1)) >0$. Thus $\# (S\setminus S\cap H)\ge 2$. Hence $\# (S)\ge n_1+2$. Now assume $\# (S)=n_1+2$. Hence 

Now assume $e(S)\ge 2$. We use induction on the integer $e(S)$. Fix $p\in S$ and set $S':= S\setminus \{p\}$. We have $e(S)-1 \le e(S')\le e(S)$. Assume
$e(S') =e(S)-1$. In this case we have $\langle \nu (S')\rangle =\langle \nu (S)\rangle$ and hence $S'$ generates $Y$. If $e(S)=e(S')$ use that the maximal dimension of a factor of the multiprojective space spanned by $S'$ is at least
$n_1-1$.\end{proof}

\begin{proposition}\label{n4}
Let $S\subset Y =(\PP^1)^k$ be a minimal and nondegenerate set with $e(S)>0$. Then: 

\quad (a) $\# (S)\ge k+e(S)+1$.

\quad (b) There is $D\in \Bb _k$ containing $S$ if and only if the $k$ ordered sets $\pi _i(S)$, $1\le i\le s$, are (with the chosen order) projectively equivalent, i.e., there is an ordering $q_1,\dots ,q_s$
of the points of $S$ and for all $i\ne j$ isomorphisms $h_{ij}: \PP^1\to \PP^1$ such that $h_{ij}(\pi _i(q_h)) =\pi _j(q_h)$ for all $h\in \{1,\dots ,s\}$.
\end{proposition}

\begin{proof}
Set $s:= \# (S)$. Since $S$ is minimal, each $\pi _{i|S}$ is injective. Assume for the moment $e(S)\ge 2$ (hence $\# (S)\ge e(S)+2)$ and take any $A\subset S$ such that $\# (A)=\# (S)-e(S)+1$. Since $S$ is minimal and $\# (A)\ge 2$, $A$ is minimal and spans $Y$. Thus to prove part (a) it is sufficient to prove that $s\ge k+2$ if $e(S)>0$.
We use induction on $k$, using in the case $k=1$ the stronger (obvious) observation that any two points of a line are linearly independent. Assume $k\ge 2$. Fix $o\in S$
and set $\{H\}:= |\Ii _o(\epsilon _k)|$. Apply Lemma \ref{ee0} to any partition of $S$ into two proper subsets. Since $S\nsubseteq H$, we have $h^1(\Ii _{S\setminus S\cap H}(\hat{\epsilon} _k)) >0$. Hence $\# (S\setminus S\cap H) \ge 2$. Since $S$ is minimal, $\eta _{k|S}$ is injective, $\eta _k(S)$ is minimal in $Y_k$ and $h^1(\Ii _{S\setminus S\cap H}(\hat{\epsilon} _k))=h^1(Y_k,\Ii _{\eta _k(S\setminus S\cap H)}(1,\dots ,1))$. 
Thus $\eta _k(S\setminus S\cap H)$ is minimal. Since $\# (\pi _i(S)) =s$ for all $i$ and $\# (S\setminus S\cap H) =2$, $\eta _k(S)$ generates $Y_k$. The inductive assumption gives $\# (S\setminus S\cap H) \ge k+1$, i.e., $s\ge k+2$. 

Now we prove part (b). First assume the existence of $D\in \Bb _k$ containing $S$. Write $D = h(\PP^1)$ with $h =(h_1,\dots ,h_k)$. Each $\pi
_{i|D}: D \to \PP^1$ is an isomorphism. Hence $\pi _{i|D}\circ \pi _{j|D}^{-1}: \PP^1\to \PP^1$ is an isomorphism sending $\pi
_j(S)$ onto $\pi _S(S)$. Now we assume that all  sets $\pi _i(S)$, $1\le i\le s$, are projectively equivalent. We order the
points $q_1,\dots ,q_s$ of $S$. For any $i\ge 2$ let $f_i: \PP^1\to \PP^1$ be the isomorphism sending $\pi _1(q_x)$ to $\pi
_j(q_x)$ for all $x\in \{1,\dots ,s\}$. We take the target of $\pi _1$ as the domain, $\PP^1$, of the morphism $h =(h_1,\dots
,h_k):
\PP^1\to Y$ we want to construct. With this condition $h_1$ is the identity map. Set $h_i:= f_i$, $i=2,\dots ,k$. By
construction and the choice of $h_1$ as the identity map, we have $h(q_x) =q_x$ for all $x\in \{1,\dots ,s\}$. We have $D:=
h(\PP^1)\in \Bb _k$ and $S\subset D$.
\end{proof}

\begin{proposition}\label{n4.00}
Set $Y: =  \PP^{n_1}\times \cdots \times \PP^{n_k}$, $n_i>0$ for all $i$, such that $k\ge 2$. Set $m:= \max \{n_1,\dots
,n_k\}$. Fix a positive integer $e$. The integer $m+k+e$ is the minimal cardinality of a minimal and nondegenerate set
$S\subset Y$ such that
$e(S)=e$.
\end{proposition}

\begin{proof}
With no loss of generality we may assume $n_1=m$. Fix $Y'\subset Y$ with $Y'=(\PP^1)^k$ and take
$S'\subset S$ with
$\# (S')=k+e+1$, $S'$ nondegenerate and minimal (we may take as $S'$ $k+e+1$ points on any $C \in \Bb _k$). Then add
$m-1$ sufficiently general points of $Y$. 

To prove the minimality of the integer $k+e+1$ one can use induction on the integer $m+k$ using linear projections in the
single factors $\ell _{\{o\},i}$ from points of $S$ (Section \ref{Sl}). The starting point of the induction is Proposition
\ref{n4}.
\end{proof}

\begin{question}
Which are the best lower bounds for $\# (S)$ in Proposition \ref{n4.00} if we also impose that $S$ has no inessential
points (resp. it is strongly essential)?
\end{question}

\section{Linearly dependent subsets with low cardinality}\label{Sc}
Unless otherwise stated $Y =  \PP^{n_1}\times \cdots \times \PP^{n_k}$, $n_i>0$ for all $i$, and $S\subset Y$ is nondegenerate
(i.e., $S\nsubseteq Y'$ for any multiprojective space $Y'\subsetneq Y$)  and linear dependent. Since $\nu $ is an embedding,
we have $\# (S)\ge 3$.

\begin{remark}\label{e1}
Assume $\# (S) =3$, i.e., assume that $\nu (S)$ spans a line. Since $\nu (Y)$ is cut out by quadrics, we have $\langle \nu
(S)\rangle \subseteq \nu (Y)$. Since $S$ is nondegenerate, we have $k=1$ and $n_1=1$.
\end{remark}

\begin{proposition}\label{e2}
Let $S\subset Y = \prod _{i=1}^{k}\PP^{n_i}$, $n_i>0$ for all $i$, be a nondegenerate circuit. Assume $\# (S) =4$. Then either
$k=1$ and
$n_1=2$ or
$k=2$ and $n_1=n_2=1$. 

Assume $k=2$. In this case $S$ is contained in a unique $D\in |\Oo _Y(1,1)|$. 

\quad (a) Assume that $D$ is reducible, say $D=D_1\cup D_2$
with $D_1\in |\Oo _Y(1,0)|$ and $D_2\in |\Oo _Y(0,1)|$, then $\# (S\cap D_1) =\# (S\cap D_2) =2$ and $D_1\cap D_2\cap S=\emptyset$. For any choice
of $A\subset S$ with $\# (A\cap D_1) = \# (A\cap D_2)=1$, the set $\langle \nu (A)\rangle \cap \langle \nu (S\setminus A)\rangle$ is a single point, $q$. We have $r_X(q) =2$ and $A, S\setminus A\in \Ss (Y,q)$.

\quad (b) Assume that $D$ is irreducible. For any choice of a set $A\subset S$ such that $\# (A)=2$ the set $\langle \nu (A)\rangle \cap \langle \nu (S\setminus A)\rangle$ is a single point, $q$. We have $r_X(q) =2$ and $A, S\setminus A\in \Ss (Y,q)$.
\end{proposition}

\begin{proof}
If $k=1$, then $S$ is formed by $4$ coplanar points, no $3$ of them collinear. The minimality assumption of $Y$ gives $n_1=2$.

From now on we assume $k\ge 2$. Take $A\subset S$ such that $\# (A) =2$ and set $B:= S\setminus A$. Since $h^1(\Ii _S(1,\dots ,1)) >0$ and $h^1(\Ii _{S'}(1,\dots ,1)) =0$ for all $S'\subsetneq S$, the set $\langle \nu (A)\rangle \cap \langle \nu (B)\rangle$ is a single point, $q$. Since $\nu (S)$ is a circuit, we have $q\notin \nu (S)$. If $q\notin \nu (Y)$, then $A$ shows that $r_X(q) =2$. 

\quad (a) Assume $q\in \nu(Y)$, say $q =\nu (o)$. Since $S$ is a circuit, we saw that $o\notin S$. Since $\# (\langle \nu (A)\rangle \cap \nu (Y))\ge 3$, $\# (\langle \nu (B)\rangle \cap \nu (Y))\ge 3$ and $\nu (Y)$ is cut out by quadrics, $\langle \nu (A)\rangle \cup \langle \nu (B)\rangle \subset \nu (Y)$, i.e., there are integers $i, j\in \{1,\dots ,k\}$
such that $\# (\pi _h(A)) = \# (\pi _m(B)) =1$ for all $h\ne i$ and all $m\ne j$. Since $\nu (S)\subset \langle \nu (A)\rangle \cup \langle \nu (B)\rangle$ and $\nu (o)\in \langle \nu (A)\rangle \cap \langle \nu (B)\rangle$, the minimality of $Y$ gives $k=2$, $n_1=n_2=1$ and that we are in case (a).

\quad (b) From now on we assume $q\notin \nu (Y)$.

\quad \emph{Claim 1:} Each $\pi _{i|S}: S\to \PP^{n_i}$ is injective.

\quad \emph{Proof of Claim 1:} Assume that some $\pi _{i|S}$ is not injective, say $\pi _{1|S}$ is not injective. Thus there
is $A'\subset S$ such that $\# (A') =2$ and $\# (\pi _1(A'))=1$. Set $B':= S\setminus A'$. We see that $A', B'$ are as
in case (a) and in particular $k=2$ and $n_1=n_2=2$ and $S=A'\cup B'\subset D_1\cup D_2$.

\quad \emph{Claim 2:} We have $n_i=1$ for all $i$.

\quad \emph{Proof of Claim 2:} Assume the existence of $i\in \{1,\dots ,k\}$ such that $n_i\ge 2$. Since $h^0(\Oo _Y(\epsilon
_i)) \ge 3$, there is $H\in |\Oo _Y(\epsilon _i)|$ containing $A$. The minimality of $Y$ and the inclusion $A\subset H$
implies $B\nsubseteq H$, i.e., $B\setminus B\cap H \ne A\setminus A\cap H$. Lemma \ref{ee0} implies $h^1(\Ii _{B\setminus B\cap
H}(\hat{\epsilon}_i)) >0$. Since $\Oo _Y(\hat{\epsilon}_i)$ is a globally generated line bundle, we get $\# (B\setminus
B\cap H)>1$. Thus $B\cap H=\emptyset$. Since any Segre embedding is an embedding, we get $\# (\pi _h(B)) =2$ for all $h\ne
i$, contradicting Claim 1 because $k\ge 2$.

\quad \emph{Claim 3:} We have $k=2$ and $n_1=n_2=1$.

\quad \emph{Proof of Claim 3:} By Claim 2 it is sufficient to prove that $k=2$. Assume
$k\ge 3$. Since $h^0(\Oo _Y(\epsilon _1)) =2$, there are $H_i\in |\Oo _Y(\epsilon _i)|$, $i=1,2$, such that $A\subset H_1\cup
H_2$. Since $k\ge 3$, as in the proof of Claim 2 we get that either $B\subset H_1\cup H_2$ or $B\cap (H_1\cup H_2) =0$ and
$\# (\pi _i(B)) =1$ for all $i\ge 3$. The latter possibility is excluded by Claim 1. Thus $S\subset H_1\cup H_2$. Hence
there
is $i\in \{1,2\}$ such that $\# (H_i\cap S) \ge 2$, i.e., $\pi _{i|S}$ is not injective, contradicting Claim 1.
\end{proof}

\begin{remark}
In case (b) of Proposition \ref{e2} $S$ is minimal, while $S$ in case (a) is not minimal.
\end{remark}

Case (a) of Proposition \ref{e2} may be generalized to the following examples of circuits in $\PP^{n_1}\times \PP^{n_2}$.

\begin{example}
Take $k=2$, i.e., $Y =\PP^{n_1}\times \PP^{n_2}$. Fix $o=(o_1,o_2)\in Y$ and set $D_i:= \pi _i^{-1}(o_i)$, $i=1,2$. We have
$D_i\in |\Oo _Y(\epsilon _i)|$. Fix $S_i\in D_i\setminus \{o\}$ such that $\# (S_i)=n_{3-i}+1$ and $\nu (S_i)$ spans the
linear space $\nu (D_i)$. Set $S:= S_1\cup S_2$. Since $o\notin S_1\cup S_2$, we have $\# (S) = n_1+n_2+2$. It is easy to
check that
$S$ is a circuit. We get in this way an irreducible and rational $(n_1^2+n_2^2+2n_1+2n_2+2)$-dimensional family of circuits.
\end{example}

\begin{remark}\label{e3.0}
Let $S\subset Y =\PP^{n_1}\times \cdots \times \PP^{n_k}$ be a nondegenerate circuit with cardinality $5$.

\quad (a) Assume $k=1$ and $n_1=3$. In this case we may take as $S$ any set with cardinality $5$ such that all its proper
subsets are linearly independent.

\quad (b) Assume $k=2$ and $n_1=n_2=1$. Since $r=3$, in this case we may take as $S$ any set with cardinality $5$ such that all
its proper subsets are linearly independent. 
\end{remark}

\begin{proposition}\label{e3.01}
Take $S\subset Y$ with $\# (S)=5$ and $e(S)\ge 2$. Let $A$ be the kernel of $S$ and let $Y' = \PP^{m_1}\times \cdots \times
\PP^{m_s}$, $s\ge 1$, be the minimal multiprojective space containing $A$. Then $A$, $e(S)$ and $Y$' are in the following list
and  all the numerical values in the list are realized by some $S$:
\begin{enumerate}
\item $e(S)=3$, $A=S$, $s=1$ and $m_s=1$;
\item $e(S)=2$, $A=S$, $s=1$ and $m_s=2$;
\item $e(S)=2$, $A=S$, $s=2$ and $m_1=m_2=1$;
\item $e(S)=2$, $\# (A)=4$, $s=1$ and $m_s=1$.
\end{enumerate}
\end{proposition}

\begin{proof}
We have $e(A) =e(S)$. Since $\nu (Y)$ is cut out by quadrics, each line $L\subset \PP^r$ containing at least $3$ points of $\nu (Y)$ is contained in
$\nu (Y)$ and hence $L =\nu (J)$ for some line $J$ in one of the factors of $Y$. 

First assume $\# (A)\le 4$. Since $e(A)\ge 2$ and $\nu $ is an embedding,
we see that $\# (A)=4$, $e(S)=2$, $s=1$ and $m_1=1$.

From now on we assume $\# (A) =5$, i.e., $A=S$.

 Since
$\nu
$ is an embedding, we see that $e(A)\le 3$ and that $e(A)=3$ if and only if $s=1$
and $m_1=1$. Now assume $e(A) =2$. In this case $\langle \nu (A)\rangle$ is a plane $\Pi$ containing at least $5$ distinct
points of $\nu (A)$. The plane $\Pi$ is contained in $\nu (Y)$ if and only if $s=1$ and $m_1=2$. Now assume that $\Pi$ is not
contained in $\nu (Y)$. Thus $s\ge 2$. Since $\nu (Y)$ is cut out by quadrics, $\nu(A)\subset \Pi$,  and $\#
(A)=5$,
$\Pi \cap \nu (Y)$ is a conic, $T$. Since $e(A)=2$ and $A$ is a finite set, $T$ is not a double line. Taking suitable $E\subset
T$ with $\# (E)=4$ and applying Proposition \ref{e2} we get $s=2$ and $m_1=m_2=1$.
\end{proof}

\begin{example}\label{p2p1}
Take $Y = \PP^2\times \PP^1$. Let $S\subset Y$ be a nondegenerate circuit with $\# (S)=5$. Here we describe the elements of
$D\in |\Oo _Y(1,1)|$ and the intersections of two of them containing $S$. We have $\deg (\nu (Y))=3$ (e.g., use that by the
distributive law the intersection number $\Oo _Y(1,1)\cdot
\Oo _Y(1,1)\cdot \Oo _Y(1,1)$ is $3$ times the integer $\Oo _Y(1,0)\cdot \Oo _Y(1,0)\cdot \Oo _Y(1,0) =1$). Since $S$ is a
circuit, $\# (S)=5$ and $r=5$, we have
$h^0(\Ii _S(1,1)) =2$.

\quad \emph{Claim 1:} The base locus $B$ of $|\Ii _S(1,1)|$ contains no effective divisor.

\quad \emph{Proof of Claim 1:} Assume that $D$ is an effective divisor contained in $B$. Since $h^0(\Ii _S(1,1))>1$, we have $D\in
|\Oo _Y(\epsilon _i)|$ for some $i=1,2$. By assumption we have $h^0(\Ii _{S\setminus S\cap D}(\hat{\epsilon}_i)) = h^0(\Ii
_S(1,1)) =2$. Since
$D$ is a multiprojective space, we have $i=2$ and $\# (S\setminus S\cap D) =1$. Lemma \ref{ee0} gives $h^1(\Ii _{S\setminus S\cap D}(1,0)) >0$, a contradiction.

By Claim 1 $S$ is contained in a unique complete intersection of two elements of $|\Oo _Y(1,1)|$.

Suppose $C$ is the complete intersection
of two elements of $|\Oo _Y(1,1)|$ (we allow the case in which $C$ is reducible or with multiple components). We know that $\deg ({C}) =\deg (\nu (Y)) =3$. Since $h^1(\Oo _Y(-2,-2))=0$ (K\"{u}nneth), a standard exact sequence gives $h^0(\Oo _C)=1$. Thus $C$ is connected. Since $C$ is the complete intersection of $2$ ample divisors, we have $\deg (\Oo _C(1,0)) =\Oo _Y(1,0)\cdot \Oo _Y(1,1)\cdot \Oo _Y(1,1)) = 2$ (where the second term is the intersection product)
and $\deg (\Oo _C(0,1)) = 1$. Since $\omega _Y\cong \Oo _Y(-3,-2)$, the adjunction formula gives $\omega _C \cong \Oo _Y(-1,0)$. Thus its irreducible components $T$
of $D$ with $\pi _1(T)$ not a point is a smooth rational curve of degree $\le 2$, while all irreducible components of $T$ of $C$ with $\pi _1(T)$ a point has arithmetic genus $1$, hence no such $T$ exists.

Now we impose that $S\subset C$. Since $h^0(\Ii _S(1,1)) =2 =h^0(\Ii _C(1,1))$, we have $\langle \nu (S)\rangle = \langle \nu ({C})\rangle \cong \PP^3$. If $C$ is irreducible, then it a smooth rational normal curve of $\PP^3$. In this case any $5$ points of $C$ forms a circuit. Of course, the general complete intersection of two elements of $|\Oo _Y(1,1)|$ is irreducible. In this
way we get an irreducible family of dimension  $13$ of circuits. Now assume that $C$ is not irreducible. Since it is connected
and $S\subset C_{\red}$, we get that
$C$ has only multiplicity one components, so it is either a connected union of $3$ lines with arithmetic genus $0$ or a union of a smooth conic and a line meeting exactly at one point and quasi-transversally. Since $\nu(S)$ is a circuit, each line contained in $\nu({C})$ contains at most $2$ points of $\nu(S)$ and each conic contained in $\nu({C})$ contains at most $3$ points of $S$. Thus if $C =T_1\cup L_1$ with $\nu(T_1)$ a smooth conic, we have $S\cap T_1\cap L_1 =\emptyset$, $\# (S\cap T_1)=3$ and $\# (S\cap L_1)=2$. Conversely any $S\subset T_1\cup L_1$ with these properties gives a circuit. The smooth conics, $T$, contained in $Y$ are of two types: either $\pi _2(T)$ is a point
(equivalently, $\pi _1(T)$ is a conic) or $\pi _2(T) =\PP^1$ (equivalently, $\pi _1(T)$ is a line; equivalently the minimal
segre containing $T$ is isomorphic to $\PP^1\times \PP^1$). Now take
$C = L_1\cup L_2\cup L_3$ with
$\nu (L_i)$ a line for all
$i$. We have
$\# (S\cap L_i) \le 2$ for all $i$. For any $1\le i< j \le 3$ such that $L_i\cap L_j\ne \emptyset$ we have $\# (S\cap
L_i\cup L_j)\le 3$. Conversely any $S$ satisfying all these inequalities gives a circuit.
\end{example}

\begin{lemma}\label{p1p1p1}
Let $\Sigma$ be the set of all nondegenerate circuits $S\subset Y:= \PP^1\times \PP^1\times \PP^1$ such that $\# (S)=5$. Let $\Bb$ be the set of all integral curves $C\subset Y$ of
tridegree $(1,1,1)$.

\quad (a) $\Sigma$ is non-empty, irreducible, $\dim \Sigma = 11$ and (if $K$ is an algebraically closed field with
characteristic $0$) $\Sigma$ is rationally connected.

\quad (b) $\Bb$ is irreducible and rational, $\dim \Bb = 6$. For any $C\in \Bb$ we have $\dim \langle \nu ({C})\rangle=3$ and
the curve
$\nu ({C})$ is a rational normal curve of $\langle \nu ({C})\rangle$.

\quad ({c}) Each $S\in \Sigma$ is contained in a unique $C\in \Bb$.

\quad (d) For any $C\in \Bb$ any $S\subset C$ with $\# (S)=5$ is an element of $\Sigma$.
\end{lemma}

\begin{proof}
By Remark \ref{up1} $\Bb$ is parametrized by the set of all triple $(h_1,h_2,h_3)\in \mathrm{Aut}(\PP^1)^3$, up to a
parametrization, i.e. we may take as $h_1$ the indentity map $\PP^1\to \PP^1$. Thus
$\Bb$ is irreducible, rational and of dimension $6$. Hence part (b) is true. 

\quad (a) Fix $C\in \Bb$. Since $C$ is irreducible and with
tridegree
$(1,1,1)$, it  is smooth and rational, $\deg (\nu ({C}))=3$ and $C$ is not contained in a proper multiprojective subspace of
$Y$. Since
$\nu({C})$ is smooth and rational and $\deg (\nu ({C}))=3$, we have $\dim \langle \nu ({C})\rangle =3$. Thus $\nu ({C})$ is a
rational normal curve of $\langle \nu ({C})\rangle$. Hence any $A\subset Y$ such that $\# (A)=5$ and
$\nu (A) \subset \nu ({C})$ is a circuit. To conclude the proof of part (d) it is sufficient to prove that $A$ is not
contained in a proper multiprojective subspace of $Y$. We prove that no $B\subset C$ with $\# (B)= 2$ is contained
in a proper multiprojective subspace of $Y$. Take $B\subset C$ such that $\# (B)=2$. Since $C$ has multidegree
$(1,1,1)$
each $\pi _{i|C}$ is injective. Thus $\# (\pi _i(B)) =2$ for all $i$, i.e., there is no $D\in |\Oo _Y(\epsilon _i)|$
containing $B$. If we prove part ({c}) of the lemma, then part (a) would follows, except for the rational connectedness of
$\Sigma$. the rational connected over an algebraically closed field with characteristic $0$ follows immediately from
\cite[Corollary 1.3]{ghsd}).

\quad (b) Fix $S\in \Sigma$. We have $\dim \langle \nu (S)\rangle =3$. Obviously $\# (L\cap \nu (S))\le 2$ for any
line
$L\subset \PP^7$ and $\# (C\cap \nu (S))\le 3$ for any plane curve
$C\subset \PP^7$. Take any $E\subset S$ such that $\# (E) =2$ and set $F:= S\setminus E$. Since $S$ is a circuit,
$\langle \nu (E)\rangle \cap \langle \nu (F)\rangle$ is a single point, $q$, and $q\notin \langle \nu (G)\rangle$ for any
$G\subsetneq A$ and any $G\subsetneq F$. Fix
$i\in
\{1,2,3\}$. Since $\dim |\Oo _Y(\hat{\epsilon}_i)|=3$, there is $D\in |\Ii _B(\hat{\epsilon}_i)|$. Lemma \ref{ee0} gives that
either $S\subset D$ or $h^1(\Ii _{S\setminus S\cap D}(\epsilon _i)) >0$. If $h^1(\Ii _{S\setminus S\cap D}(\epsilon _i)) >0$
we have $S\setminus S\cap D = E$ and $\# (\pi _h(E)) =1$ for all $h\in \{1,2,3\}\setminus \{i\}$, i.e., $\# (\eta
_i(E)) =1$.

Fix $i\in \{1,2,3\}$ and $A\subset S$ such that $\# (A)=\# (\eta _i(A)) =2$. Set $B:= S\setminus A$. We just
proved that any $D_i\in |\Oo _Y(\hat{\epsilon}_i)|$ containing $B$ contains $S$. Since $\# (L\cap \nu (S)\rangle)\le 2$
for any line
$L\subset \PP^7$, we get $\# (\eta _i(S))\ge 2$ for all $i=1,2,3$. Thus there are $D_i\in |\Oo _Y(\hat{\epsilon}_i)|$,
$i=1,2,3$ such that $S\subseteq D_1\cap D_2\cap D_3$.
Since $h^0(\Oo _Y(\epsilon _i)=2$, a standard exact sequence gives $\dim \langle \nu (D_i)\rangle =5$. Since $\langle
\nu(D_i)\cup
\langle \nu (D_j\rangle =\PP^7$ for all $i\ne j$, the Grassmann's formula gives $\dim \langle \nu (D_i)\rangle
\cap  \langle \nu (D_j)\rangle$. Thus for any $G\in \Sigma$ contained in $D_i\cap J$ we have $\langle \nu (G)\rangle =\langle \nu (D_i)\rangle
\cap  \langle \nu (D_j))\rangle$. Assume for the moment that $D_1$ is reducible, say $D_1 =D'\cup D''$ with $D'\in |\Oo
_Y(\epsilon _2)|$ and $D''\in |\Oo _Y(\epsilon _3)|$. Take any $M\in |\Oo _Y(\epsilon _2)|$ containing no irreducible
component of $D_1$.

\quad \emph{Claim 1:} There is no line $L\subset \nu (Y)$ such that $\# (L\cap \nu (S) )=2$.

\quad \emph{Proof of Claim 1:} Assume that $L$ exists and take $A\subset S$ such that $\# (A)=2$ and $\nu (A)\subset L$.
By the structure of lines of $\nu (Y)$ there is $i\in \{1,2,3\}$ such that $\# (\pi _h(A)) =1$ for all $h\in
\{1,2,3\}\setminus \{i\}$, i.e., $\# (\eta _i(A))=1$. Set $\{M\}:=  |\Ii _A(\epsilon _i)|$. Since $M$ is a multiprojective
space, we have $S\nsubseteq M$. Thus Lemma \ref{ee0} gives $h^1(\Ii _{S\setminus S\cap M}(\hat{\epsilon}_i)) >0$. Thus one of
the following cases occur:
\begin{enumerate}
\item there are $u, v \in S\setminus S\cap M$ such that $u\ne v$ and $\eta _i(u)=\eta _i(v)$;
\item we have $\# (S\setminus S\cap M) = \# (\eta _i(S\setminus S\cap M)) =3$ and $\nu _i(\eta _i(S\setminus S\cap
M))$ is contained in a line.
\end{enumerate}

First assume the existence of $u, v$. We get the existence of a line $R\subset \nu (Y)$ such that $\{u,v\}\subset R$ and
$R\cap L=\emptyset$. Thus $\dim \langle L\cup R\rangle =3$. Since $L$ and $R$ are generated by the points of $\nu (S)$
contained in them, we get $\nu (S)\subset \langle L\cup R\rangle$. Set
$\{o\}:= S\setminus (A\cup \{u,v\})$. Since any line of $\PP^7$ contains at most $2$ points of $\nu (S)$, we get $\nu (o)\in 
\langle L\cup R\rangle \setminus (L\cup R)$. Thus there is a unique line $J\subset \langle L\cup R\rangle$ such that $J\cap
L\ne \emptyset$, $R\cap J\ne \emptyset$ and $\nu (o)\in J$. Since $\nu (Y)$ is cut out by quadrics and $\{o\}\cup L\cup
R\subset \nu (Y)$, we get $J\subset \nu (Y)$. Since $\langle L\cup R\rangle \cap \nu (Y)$ is cut out by quadrics and $\nu (Y)$
contains no plane, we get that $\langle L\cup R\rangle \cap \nu (Y)$ is a quadric surface $Q$. 

\quad \emph{Subclaim:} $Q = \nu (Y))$ for some multiprojective subspace $Y'\subsetneq Y$.

\quad \emph{Proof of the subclaim:} Since the irreducible quadric surface $Q$ contains the lines $R$ and $L$ such that $R\cap
L=\emptyset$, we have
$Q\cong
\PP^1\times \PP^1$ and all elements of the two rulings of $Q$ are embedded as lines. The structure of linear spaces contained
in Segre varieties gives
$Q =\nu (Y')$ for some
$2$-dimensional multiprojective subspace $Y'\subset Y$.

 Since $S\subset \langle L\cup R\rangle \cap \nu (Y) =Q$, the subclaim gives a contradiction.

Now assume  $\# (S\setminus S\cap M) = \# (\eta _i(S\setminus S\cap M)) =3$ and that
$\nu _i(\eta _i(S\setminus S\cap M))$ is contained in a line. Thus there is $j\in \{1,2,3\}\setminus \{i\}$ such that $\#
(\pi _j(\eta _i(S\setminus S\cap M)))=1$. Since $j\ne i$, we have $\pi _j((\eta _i(S\setminus S\cap
M))) = \pi _j(S\setminus S\cap M)$. Thus there is $W\in |\Oo _Y(\epsilon _j)|$ such $S\setminus S\cap M$.
Since $W$ is a multiprojective space, we have $S\nsubseteq W$. Thus Lemma \ref{ee0} gives $h^1(\Ii _{S\setminus S\cap
W}(\hat{\epsilon}_j)) >0$. Since $\Oo _Y(\hat{\epsilon}_j)$ is globally generated and $S\setminus S\cap W \subseteq A$, we
get $S\setminus S\cap W =A$. Since $j\ne i$ we have $\# (\eta _j(A)) =2$. Thus $h^1(\Ii _A(\hat{\epsilon}_j)) =0$, a
contradiction.

\quad ({c}) In step (b) we saw that $S\subseteq D_1\cap D_2\cap D_j$ for some $D_i\in |\Oo _Y(\hat{\epsilon}_i)|$ and that
$\langle \nu (S)\rangle = \langle \nu (D_i)\rangle \cap \langle \nu (D_j)\rangle$. Now we assume that both $D_1$ and $D_2$
are reducible and that they have a common irreducible component. Write $D_1 = D'\cup D''$ with $D'\in |\Oo _Y(\epsilon _3)|$
and $D''\in |\Oo _Y(\epsilon _2)|$. We see that $D_1$ and $D_2$ have $D'$ as their common component, say $D_2 = D'\cup M$
with $M\in |\Oo _Y(\epsilon _1)|$. The curve $\nu (D''\cap M)$ is a line. Since $\dim \langle \nu (D')\rangle =3$,
we have $\langle \nu (D_1)\rangle \cap \langle \nu (D_2)\rangle = \langle \nu (D')\rangle$. Since $\mu (Y)$ is cut out by
quadrics
and contains no $\PP^3$, we have $M\cap D''\subset D'$. Thus $S\subset D'$, a contradiction.

In the same way we exclude the existence of a surface contained in $D_i\cap D_j$ for any $i\ne j$.

Now assume that $D_1 =D'\cup D''$ is reducible, $D_2 = M'\cup M''$ is reducible, but that $D_1$ and $D_2$ have no common
irreducible component. We get that $\nu (D_1\cap D_2)$ is a union of $4$ lines. Since $S\subset D_1\cap D_2$, Claim 1 gives a
contradiction. Thus at most one among $D_1$, $D_2$ and $D_3$ is reducible. Assume that $D_1$ is reducible, say  $D_1 = D'\cup D''$ with $D'\in |\Oo _Y(\epsilon _3)|$
and $D''\in |\Oo _Y(\epsilon _2)|$. The curve $\nu (D_2\cap D'')$ is a conic (maybe reducible) and hence it contains at most
$3$ points of $\nu (S)$. The curve
$\nu (D_2\cap D')$ is a line. Since $S\subset D_1\cap D_2$, Claim 1 gives a contradiction.
Thus each $D_i$, $1\le i\le 3$, is irreducible.

\quad (d) By part ({c}) $S\subseteq D_1\cap D_2\cap D_3$ with $D_i\in |\Oo _Y(\hat{\epsilon}_i)|$ and each $D_i$ irreducible.
Thus $T:= D_1\cap D_2$ has pure dimension $1$. We have $\Oo _Y(1,1,1)\cdot \Oo _Y(0,1,1)\cdot \Oo _Y(1,0,1)=3$ (intersection
number), i.e., the curve $\nu (T)$ has degree $3$. Since $S\subset T$ and no line contains two points of $\nu (S)$ (Claim 1),
$T$ must be irreducible. Since $S\subset T$, no proper multiprojective subspace of $Y$ contains $T$. Thus $T$ has multidegree
$(a_1,a_2,a_3)$ with $a_i>0$ for all $i$. Since $\deg (T) =a_1+a_2+a_3$, we get $a_i=1$ for all $i$, i.e., $T\in
\Bb$. Fix
$C, C'\in
\Bb$ such that
$C\ne C'$ and assume
$S\subseteq C\cap C'$. We have $\langle \nu({C})\rangle =\langle \nu (S)\rangle = \langle \nu (C')\rangle$. Hence $\langle \langle \nu (S)\rangle \cap \nu (Y)$
contains two different rational normal curves of $\PP^3$ with $5$ common points. Such a reducible curve $\nu ({C})\cup \nu (C')$ is contained in a unique quadric surface $Q'$
and $Q'$ is integral. Since $\nu (Y)$ is cut out by quadrics, there is a an integral surface $G\subset X$ such that $\nu (G)=Q'$. Since $\deg (G)=2$, $G\in |\Oo _Y(\epsilon _i)|$
for some $i$. Since $G$ is a multiprojective space and $S\subset G$, we got a contradiction.
\end{proof}

\begin{proof}[Proof of Theorem \ref{e3}:]
The cases $k=1$ and $k=2$, $n_1=n_2=1$ are obvious (the latter because it has $r=3$). Thus we may assume $k\ge 2$ and
$n_1+\cdots +n_k\ge 3$.

\quad (a) Assume $k=2$, $n_1=2$ and $n_2=1$. Thus $r=5$. All $S\in \Sigma$ are described in Example \ref{p2p1}. The same proof
works if $k=2$,
$n_1=1$ and
$n_2=2$.

\quad (b) Assume $k=2$ and $n_1=n_2=2$. Fix $a, a'\in S$, $a\ne a'$ and take
$H\in |\Oo _Y(1,0)|$ containing $\{a,a'\}$. Since $S\nsubseteq H$, Lemma \ref{ee0} gives $h^1(\Ii _{S\setminus
S\cap H}(0,1)) >0$. Thus either there are $b, b'\in S\setminus S\cap H$ such that $b\ne b'$ and $\pi _2(b)=\pi _2(b')$
or $\# (S\setminus S\cap H) =3$ and $\pi _2(S\setminus S\cap H)$ is contained in a line. 

First assume the existence of $b, b'$. Write $S =\{b,b',u,v,w\}$. We get the existence of $D\in |\Oo _Y(0,1)|$ containing
$\{b,b',u\}$. Since $S\nsubseteq D$, Lemma \ref{ee0} gives $h^1(\Ii _{S\setminus S\cap D}(1,0)) >0$. Thus $S\setminus S\cap D
= \{v,w\}$ and $\pi _2(v)=\pi _2(w)$. Taking $w$ instead of $u$ we get $\pi _2(v)=\pi _2(v)=\pi _2(w_2)$. Hence we get
$U\in |\Oo _Y(0,1)|$ containing at least $4$ points of $S$. Since $U$ is a multiprojective space, we have $S\nsubseteq U$.
Lemma \ref{ee0} gives $h^1(\Ii _{S\setminus S\cap U}(1,0))>0$. Since $\# (S\setminus S\cap U)=1$, we get a contradiction.

Now assume that $\pi _2(S\setminus \{a,a'\})$ is contained in a line. Thus there is $M\in
|\Oo _Y(0,1)|$ containing $S\setminus \{a,a'\}$. Since $S\nsubseteq M$, Lemma \ref{ee0} gives $h^1(\Ii _{S\setminus S\cap
M}(1,0))>0$. Since $S\setminus S\cap M\subseteq \{a,a'\}$, we get $\pi _1(a)=\pi _1(a')$. We conclude as we did with
$\{b,b'\}$ using the other factor of $\PP^2\times \PP^2$.

\quad ({c}) Assume $k=2$, $n_1=3$ and $n_2=1$. Take $H\in |\Oo
_Y(1,0)|$ containing $B$. Since $H$ is a multiprojective space, we have $S\nsubseteq H$. Thus Lemma \ref{ee0} gives
$h^1(\Ii _{A\setminus A\cap H}(0,1)) >0$. Since $\Oo _{\PP^1}(1)$ is very ample and $\# (A\setminus A\cap H)\le 2$, we get
$A\cap H =\emptyset$ and $\# (\pi _2(A)) =1$. Set  $\{M\}:=  |\Ii_A(0,1)|$. Since $S\nsubseteq M$, Lemma \ref{ee0}
gives $h^1(\Ii _{B\setminus B\cap H}(1,0)) >0$. Thus either there are $b, b'\in B\setminus B\cap H$ such that $b\ne b'$
and $\pi _1(b)=\pi _1(b')$ or $B\cap H =\emptyset$ and $\pi _1(B)$ is contained in a line. 

Assume the existence of $b, b'$. Since $h^0(\Oo _Y(1,0))=4$, we get the existence if $D\in |\Oo _Y(1,0)|$ contained $B$ and a
point of $A$. Take any $D'\in |\Oo _Y(0,1)|$ such that $D'\cap S=\emptyset$. We have $\# ((D\cup D')\cap S)=\# (D\cap
S) \ge 4$. Since $S$ is a circuit and $\# (D\cap
S) \ge 4$, we get $D\cup D'\supset S$ and hence $D\supset S$. Since $D$ is a multiprojective space, we get a contradiction.
Now assume that $\pi _1(B)$ is contained in a line. Since $h^0(\Oo _Y(1,0)) =4$, we get the existence of $D''\in |\Oo
_Y(1,0)|$ containing $S$. Since $D''$ is a multiprojective space, we get a contradiction.

\quad (d) As in step ({c}) we exclude all other cases with $k=2$ and $n_1+n_2\ge 4$. Thus from now on we assume $k\ge 3$.

\quad (e) See Lemma \ref{p1p1p1} for the description of the case $k=3$ and $n_1=n_2=n_3=1$.

\quad (f) Assume $k=3$, $n_1=2$ and $n_2=n_2=1$. Fix $a, a'\in S$ such that $a\ne a'$. Take $H\in |\Oo _Y(1,0,0)|$ containing
$\{a,a'\}$. Since $S\nsubseteq H$, Lemma \ref{ee0} gives $h^1(\Ii _{S\setminus S\cap H}(0,1,1)) >0$. Thus either there are $b,
b'\in S\setminus S\cap H$ such that $b\ne b'$ and $\eta _1(b)=\eta _1(b')$ or $\# (S\setminus S\cap H) =3$ and there is
$i\in \{2,3\}$ such that $\# (\pi _i(S\setminus S\cap H)) =1$.  In the latter case there is $M\in |\Oo _Y(\epsilon _i)|$
containing $S\setminus S\cap H$. Since $S\nsubseteq M$, Lemma \ref{ee0} gives $h^1(\Ii _{S\setminus S\cap
M}(\hat{\epsilon}_i)) >0$. Thus $\# (S\setminus S\cap M) =2$, i.e., $S\setminus S\cap M = \{a,a'\}$ and $\eta _i(a) =\eta
_i(a')$. In particular we have $\pi _1(a) =\pi _1(a')$. Thus there is $H'\in |\Oo _Y(1,0,0)|$ containing $\{a,a'\}$ and at
least another point of $H$. Using $H'$ instead of $H$ we exclude this case (but not the existence of $b, b'$ for $S\setminus
H'\cap S$), because $\# (S\setminus S\cap H') \le 2$.

Now assume the existence of $b, b'\in S\setminus S\cap H$ such that $b \ne b'$ and $\eta _1(b) =\eta _1(b')$. Thus there is $D\in |\Oo _Y(0,0,1)|$ containing $\{b,b'\}$. Using Lemma \ref{ee0} we get $h^1(\Ii _{S\setminus S\cap D}(1,1,0)) >0$. Thus one of the following cases occurs:
\begin{enumerate}
\item $S\cap D =\{b,b'\}$, $\pi _1(S\setminus S\cap D)$ is contained in a line and $\# (\pi _2(S\setminus S\cap D)) =1$;
\item there are $x, y\in S\setminus S\cap D$ such that $x\ne y$ and $\eta _3(x) =\eta _3(y)$. 
\end{enumerate}

First assume $S\cap D =\{b,b'\}$, $\pi _1(S\setminus S\cap D)$ is contained in a line and $\# (\pi _2(S\setminus S\cap D)) =1$. There is $T\in |\Oo _Y(\epsilon _2)|$
containing $S\setminus S\cap D$. Lemma \ref{ee0} gives $h^1(\Ii _{\{b,b'\}}(1,0,1)) >0$. Thus $\eta _2(b)=\eta _2(b')$. Since
$\eta _1(b)=\eta _1(b')$, we get $b=b'$, a contradiction.

Assume the existence of $x, y$ such that $x\ne y$ and $\eta _3(x)=\eta _3(y)$. Write $S = \{b,b',x,y,v\}$. There is $W\in |\Oo _Y(1,0,0)|$ containing $b$ and $x$ and hence containing $y$. By Lemma \ref{ee0} we get $h^1(\Ii _{S\setminus W}(0,1,1)) >0$.
Since $\# (S\setminus S\cap W)\le 2$, we get $S\setminus S\cap W =\{b',v\}$ and $\eta _1(b') =\eta _1(v)$.   Using $D'\in |\Oo _Y(0,1,0)|$ containing $\{b,b'\}$ instead of $D$
we get the existence of $x', y' \in \{x,y,v\}$ such that $\eta _2(x') =\eta _2(y')$ and $x'\ne y'$. We have $\{x',y'\}\cap \{x,y\} \ne \emptyset$. With no loss of generality we may assume $x=x'$.
Either $y' =y$ or $y'=v$. If $y'=y$, we get $x=y$, a contradiction. Thus $\eta _2(v) =\eta _2(y)$. Hence $\pi _1(x)=\pi _1(y) =\pi _1(v)$. Thus there is $W'\in |\Oo _Y(1,0,0)|$ containing at least $4$ points of $S$. Take a general $W_1\in |\Oo _Y(0,1,1)|$. Thus $S\cap W_1 =\emptyset$. Since $h^1(\Ii _S(1,1,1)) >0$ and $W'\cup W_1$ contains at least $4$
points of $S$, we get $S\subset W'\cup W_1$. Thus $S\subset W'$. Since $W'$ is a proper multiprojective space of $Y$, we get a
contradiction.

\quad (g) Step (e) excludes all cases with $k\ge 3$ and $n_i\ge 2$ for at least one $i$.

\quad (h) Assume $k=4$ and $n_i=1$ for all $i$. Fix $o\in S$ and let $H$ be the only element of $|\Oo _Y(\epsilon _4)|$
containing $o$.
Since $H$ is a multiprojective space, we have $S\nsubseteq H$. Thus Lemma \ref{ee0} gives $h^1(\Ii _{S\setminus S\cap
H}(\hat{\epsilon}_4))>0$. Thus one of the following two cases occurs:
\begin{enumerate}
\item $\# (\eta _4(S)) =4$ and $h^1(Y_4,\Ii _{\eta _4(S\setminus S\cap H)}(1,1,1)) >0$;
\item there are $u, v\in S\setminus S\cap H$ such that $u\ne v$ and $\eta _4(u) =\eta _4(v)$.
\end{enumerate}

\quad (h1) Assume that case (1) occurs. In this case $S\setminus S\cap H =S\setminus \{o\}$. By Proposition \ref{e2} there is
an integer
$i\in \{1,2,3\}$ such that $\# (\pi _i(S\setminus \{o\}) =1$. Thus there is $M\in |\Oo _Y(\epsilon _i)|$ containing
$S\setminus \{o\}$. Take $W\in |\Oo _Y(\hat{\epsilon}_4)|$ such that $W\cap S =\emptyset$. Since $H\cup W$ is an element
of $|\Oo _Y(1,1,1,1)|$ containing at least $4$ points of $S$ and $h^1(\Ii _S(1,1,1,1)) >0$, we have $S\subset H\cup W$. Since
$S\cap W=\emptyset$, we get $S\subset H$, a contradiction.

\quad (h2) By step (h1) there are $u, v\in S\setminus S\cap H$ such that $\eta _4(u)=\eta _4(v)$. For any $a\in S$ let 
$H_a$ be the only element of $|\Oo _Y(\epsilon _4)|$ containing $a$. Thus $H_o =H$. By step (h1) applied to $a$ instead of $o$
there are
$u_a,v_a\in S\setminus S\cap H_a$ such that $u_a\ne v_a$ and $\eta _4(u_a) = \eta _4(v_a)$. Set $E:= \cup _{a\in S} \{u_a,v_a\}$.

\quad \emph{Claim 1:} We have $\# (E) >2$ and for all $a, b\in S$ either $\{u_a,v_a\} =\{u_b,v_b\}$ or $\{u_a,v_a\}\cap \{u_b,v_b\} =\emptyset$.

\quad \emph{Proof of Claim 1:} Assume $\# (E) \le 2$, i.e., $\{u_a,v_a	\} = \{u_b,v_b\}$. Since $x\notin \{u_x,v_x\}$, taking $b = u_{u_a}$ we get a contradiction. Fix $a, b\in S$ such that $a\ne b$ and assume $\# (\{u_a,v_a\}\cap \{u_b,v_b\} )=1$, say $\{u_a,v_a\}\cap \{u_b,v_b\} =\{u_a\}$. Taking $x:= u_a$, $y:=v_a$ and $z:= v_b$, find $x, y, z\in S$
such that $\# (\{x,y,z\})=3$ and $\nu (\{x,y,z\})$ is contained in a line of of $\nu (Y)$. Thus $S$ is not a circuit, a contradiction. 

The first assertion of Claim 1 gives $\# (E) >2$. Then the second assertion of Claim 1 gives $\# (E)\ge 4$.
There is $M\in |\Oo _Y(\epsilon _1)|$ containing $E$. Since $S\nsubseteq M$, we first get $\# (E) =4$ and then (by Lemma \ref{ee0}) $h^1(\Ii _{S\setminus S\cap M}(0,1,1,1)) >0$, contradicting the global spannedness of $\Oo _Y(0,1,1,1)$.

\quad (i) Steps (g) and (h) exclude all cases with $k\ge 4$.
\end{proof}

\end{document}